\documentclass{elsarticle}

\usepackage{natbib}
\usepackage{etex}
\usepackage[english]{babel}
\usepackage[cp1250]{inputenc}
\usepackage{amssymb,amsthm,amsfonts,amsmath}
\usepackage{mathrsfs}
\usepackage{verbatim}
\usepackage{url}
\usepackage{pictexwd,dcpic}
\usepackage{pgfplots}
\usepackage{tikz}

\renewenvironment{description}[1][0pt]
  {\list{}{\labelwidth=0pt \leftmargin=#1
   }}
  {\endlist}

\newcommand{\RR}{\mathbb R}

\newcommand{\NN}{\mathbb N}

\newcommand{\CC}{\mathbb C}
\newcommand{\QQ}{\mathbb Q}

\newcommand{\TT}{\mathbb T}

\newcommand{\fC}{C}
\newcommand{\fd}{d}

\newcommand{\fp}{c}

\newcommand{\fD}{D}

\newcommand{\cM}{\mathcal M}

\newcommand{\cK}{\mathscr K}

\newcommand{\cS}{\mathscr S}
\newcommand{\HH}{\mathbb H}

\newcommand{\im}{\mathrm{Im}}

\newcommand{\benu}{\begin{enumerate}}
\newcommand{\eenu}{\end{enumerate}}
\newcommand{\bop}{\begin{opomba}}
\newcommand{\eop}{\end{opomba}}

\newcommand{\Cl}{\text{Cl}}

\newcommand{\Pos}{\text{Pos}}

\newcommand{\beqn}{\begin{align*}}
\newcommand{\eeqn}{\end{align*}}

\newcommand{\bdefi}{\begin{definition}}
\newcommand{\edefi}{\end{definition}}
\newcommand{\bcor}{\begin{corollary}}
\newcommand{\ecor}{\end{corollary}}
\newcommand{\bthe}{\begin{theorem}}
\newcommand{\ethe}{\end{theorem}}
\newcommand{\bpro}{\begin{proposition}}
\newcommand{\epro}{\end{proposition}}
\newcommand{\blem}{\begin{lemma}}
\newcommand{\elem}{\end{lemma}}
\newcommand{\brem}{\begin{remark}}
\newcommand{\erem}{\end{remark}}
\newcommand{\bequ}{\begin{equation}}
\newcommand{\eequ}{\end{equation}}
\newcommand{\bprf}{\begin{proof}}
\newcommand{\eprf}{\end{proof}}

\newtheorem{theorem}{Theorem}[section]
\newtheorem{corollary}[theorem]{Corollary}
\newtheorem{lemma}[theorem]{Lemma}
\newtheorem{proposition}[theorem]{Proposition}

\newtheorem*{problem*}{Problem}
\newtheorem*{problem'*}{Problem'}

\newtheorem*{conjecture*}{Conjecture}

\theoremstyle{definition}
\newtheorem{definition}[theorem]{Definition}
\newtheorem{remark}[theorem]{Remark}

\newtheorem*{convention*}{Convention}

\begin{document}

%
%
%
%
%
%
%
%
%

\title{Matrix Fej\' er-Riesz Theorem with gaps}

\author{Alja\v z Zalar}
\ead{aljaz.zalar@imfm.si}

\address{%
Institute of Mathematics, Physics, and Mechanics, Jadranska 19, 1000
Ljubljana, Slovenia}

\date{\today}

\begin{abstract} 
	The matrix Fej\'er-Riesz theorem characterizes positive semidefinite matrix 
	polynomials on the real line $\RR$.
	We extend a characterization to arbitrary closed semialgebraic sets $K\subseteq \RR$ by the
	use of matrix preorderings from real algebraic geometry. 
	In the compact case a denominator-free characterization exists, while in the non-compact case there 
	are counterexamples. However, there is a weaker characterization with denominators in the non-compact 
	case.
	At the end we extend the results to algebraic curves.
\end{abstract}

\begin{keyword}
positive polynomials\sep matrix polynomials\sep preorderings\sep Nichtnegativstellensatz\sep
real algebraic geometry
\MSC 14P10\sep 13J30\sep 47A56
\end{keyword}

\maketitle

\section{Introduction}

\subsection{Motivation}



The matrix Fej\' er-Riesz theorem is the following result (For the proof see either of \cite{Gohberg-Krein}, \cite{Rosenblatt}, \cite{Jakubovic}, \cite{Djok}, \cite{Choi}, \cite{Rodman}, \cite{Drit}).

\begin{theorem} 
\label{zvezna}
	Let 
		$F(x)=\sum_{m=0}^{2N}F_m x^m$
	be a $n\times n$ matrix polynomial from $M_n(\CC[x])$ which is 
	positive semidefinite on $\RR$.
	Then there exists a matrix polynomial 
	$G(x)=\sum_{m=0}^N G_m x^m\in M_n(\CC[x])$ such that
		$F(x)=G(x)^\ast G(x)$
	where $G(x)^\ast=\sum_{m=0}^N G_m^{\ast} x^m=\sum_{m=0}^N \overline{G_m}^Tx^m=\overline{G(x)}^T$.
\end{theorem}



In the scalar case ($n=1$) Theorem \ref{zvezna} has already been extended to a finite union of 
points and intervals (not necessarily bounded) in $\RR$ by S.\ Kuhlmann and Marshall 
\cite[Theorem 2.2]{K-M}.
The main problem of our paper is the following.

\begin{problem*} 
	Characterize univariate matrix polynomials which are positive semi\-de\-fi\-nite on a finite union of 
	points and intervals (not necessarily bounded) in $\RR$.
\end{problem*}


Our main results, which will be explicitly stated in Subsection 1.3, are a denominator-free generalization of Theorem \ref{zvezna} to a finite union of compact intervals in $\RR$, a classification of counterexamples for a denominator-free generalization to an unbounded finite union of closed intervals in $\RR$ and a generalization with denominators in this case.

\subsection{Notation and known results} \label{Continuous-known-results}

Let $M_n(\CC[x])$ be a set of all $n\times n$ matrix polynomials over $\CC[x]$ equipped
with the \textsl{involution} $F(x)^\ast=\overline{F(x)}^T$ where $\overline{x}=x$. 

\begin{remark}
	For $n=1$ and $p(x):=\sum_{i=0}^m a_ix^i\in \CC[x]$, the involution is
	$p(x)^{\ast}=\sum_{i=0}^m \overline{a_i}x^i$.
\end{remark}
We say $F(x)\in M_n(\CC[x])$ is \textsl{hermitian} if $F(x)=F(x)^\ast$. We write $\HH_n(\CC[x])$ for the set of all hermitian matrix polynomials from $M_n(\CC[x])$. 
A matrix polynomial $F(x)\in \HH_n(\CC[x])$ is \textsl{positive semidefinite} 
in $x_0\in \CC$ if $v^\ast F(x_0)v \geq 0$ 
for every nonzero $v\in \CC^n$.
We denote by $\sum M_n(\CC[x])^2$ the set of all finite sums
of the expressions of the form $G(x)^\ast G(x)$ where $G(x)\in M_n(\CC[x])$. We call such expressions 
\textsl{hermitian squares} of matrix polynomials.

The \textsl{closed semialgebraic set} associated to a finite subset
	$S=\left\{g_1,\ldots, g_s\right\}\subset \RR\left[x\right]$
is given by
	$K_{S}=\left\{x\in \RR\colon g_j(x)\geq 0,\; j=1,\ldots,s\right\}.$
We define the \textsl{$n$-th matrix quadratic module generated by $S$} in $\HH_n(\CC[x])$ by
	\begin{eqnarray*}
		M^n_{S} 
			&:=& 
				\left\{\sigma_0+ \sigma_1 g_1+\ldots + \sigma_s g_s\colon \sigma_j\in 
				\sum M_n(\CC[x])^2,\;j=0,\ldots,s\right\},
	\end{eqnarray*}
and the \textsl{$n$-th matrix preordering generated by $S$} in $\HH_n(\CC[x])$ by
	\begin{eqnarray*}
		T^n_{S} 
			&:=& 
				\left\{\sum_{e\in \left\{0,1\right\}^s} \sigma_e \underline{g}^e\colon \sigma_e\in 
				\sum M_n(\CC[x])^2\;\text{for all}\;e\in \left\{0,1\right\}^s\right\},
	\end{eqnarray*}
where $e:=(e_1,\ldots,e_s)$ and $\underline{g}^e$ stands for $g_1^{e_1}\cdots g_s^{e_s}$.
\begin{remark}
	Note that $T^n_{S}$ is the quadratic module generated by all products $\underline{g}^e$, $e\in
	\{0,1\}^s$.
\end{remark}


We write $\Pos^n_{\succeq 0}(K_S)$ for the set of all $n\times n$ 
hermitian matrix polynomials which are positive semidefinite on $K_S$.
We say $M^n_S$ (resp.\ $T^n_S$) is \textsl{saturated} if 
$M^n_S=\Pos^n_{\succeq 0}(K_S)$ (resp.\ $T^n_S=\Pos^n_{\succeq 0}(K_S)$).

Theorem \ref{zvezna} can be restated in the following form.

\renewcommand{\thetheorem}{1.1'}\begin{theorem} \label{zvezni-fejer-riesz} 
	Assume the notation as above. The set $M^n_{\emptyset}=T^n_{\emptyset}$ is saturated for every $n\in\NN$.
\end{theorem}

\addtocounter{theorem}{-1}
\renewcommand{\thetheorem}{\arabic{section}.\arabic{theorem}}

The aim of this article is to study matrix generalizations of Theorem \ref{zvezni-fejer-riesz} 
to an arbitrary closed semialgebraic set $K\subseteq \RR$. In this notation Problem becomes the following.

\begin{problem'*}\label{zvezno-vprasanje}
		Assume $K\subseteq \RR$ is a closed semialgebraic set. 
	Does there exist a finite set $S\subset \RR[x]$ 
	such that $K=K_S$ and the $n$-th matrix quadratic module $M^n_S$ or preordering $T_{S}^n$ is saturated 
	for every $n\in \NN$?
\end{problem'*}

%

Now we recall a description of a closed semialgebraic set $K\subseteq \RR$, introduced in \cite{K-M}, which solves Problem' for $n=1$. A set $S=\{g_1,\ldots,g_s\}\subset \RR\left[x\right]$ is the
\textsl{natural description} of $K$ if it satisfies the following conditions:
	\begin{enumerate}
		\item[(a)] If $K$ has the least element $a$, then $x-a\in S$.
		\item[(b)] If $K$ has the greatest element $a$, then $a-x\in S$.
		\item[(c)] For every $a\neq b\in K$, if $(a,b)\cap K=\emptyset$, then $(x-a)(x-b)\in S$.
		\item[(d)] These are the only elements of $S$.
	\end{enumerate}
	
	Problem' has already been solved in the following cases:
%

\begin{enumerate}
	\item The preordering $T_S^1$ is saturated for the natural description $S$ of $K$ 
		(see \cite[Theorem 2.2]{K-M}).
	\item For $K=K_{\{x,1-x\}}=[0,1]$, $M^n_{\{x,1-x\}}$ is saturated for every $n\in \NN$ 
			(see \cite[Theorem 2.5]{ds} or \cite[Theorem 7]{Sch-Sav}).
	\item For $K=K_{\{x\}}=[0,\infty)$, $M^n_{\{x\}}$ is saturated for every $n\in \NN$ 
			(see \cite[Theorem 8]{Sch-Sav} or \cite[Proposition 3]{Cim-Zal}).
\end{enumerate}	

	Even more can be said in the case $n=1$. There is a characterization of finite sets 
$S=\{g_1,\ldots,g_s\}\subset \RR\left[x\right]$ such that the preordering $T_S^1$ is saturated, which we now explain. We separate two possibilities according to the compactness of $K_S$.
\begin{enumerate}
	\item \textsl{$K_S$ is not compact:} By \cite[Theorem 2.2]{K-M}, 
		$T^1_S$ is saturated iff $S$ contains each of the 	
		polynomials in the natural description of $K_S$ up to scaling by positive constants.
	\item \textsl{$K_S$ is compact:} Write $K_S$ as the union of pairwise disjoint points and 
		intervals, i.e., $K_S=\cup_{j=1}^t [x_j,y_j]$ where $x_j\leq y_j$ for every $j=1,\ldots,t$. 
		By a special case of Scheiderer's results \cite[Corollary 4.4]{Scheiderer4}, \cite[Theorem 
		5.17]{Scheiderer3} (which cover non-singular curves in $\RR^n$), $M_S^1=T^1_S$ and 
		$M_S^1$ is 
		saturated iff the following two conditions hold:
		\begin{enumerate}
			\item[(a)] For every left endpoint $x_j$ there exists $k\in \{1,\ldots,s\}$
				such that $g_k(x_j)=0$ and $g_k'(x_j)>0$.
			\item[(b)] For every right endpoint $y_j$ there exists $k\in \{1,\ldots,s\}$
				such that $g_k(y_j)=0$ and $g_k'(y_j)<0$.
		\end{enumerate}
		(For another proof see \cite[Theorem 3.2]{K-M-2}.).
		We call every set $S\subset \RR\left[x\right]$ which satisfies the two conditions above
		a \textsl{saturated description} of $K_S$.
	\end{enumerate}
	
\begin{convention*}
	An interval always has a non-empty interior.
\end{convention*}

\subsection{New results} \label{Problem 2 - novi rezultati}

One of the main results of the paper which solves Problem' for compact sets $K$ 
is the following.

\renewcommand{\thetheorem}{C}\begin{theorem} 
	Let $K$ be a compact semialgebraic set.
	The $n$-th matrix quadratic module $M_S^n$ is saturated for every $n\in \NN$
	iff $S$ is a	saturated description of $K$ (see Theorem \ref{nasicenost-na-R}).
\end{theorem}

\addtocounter{theorem}{-1}
\renewcommand{\thetheorem}{\arabic{section}.\arabic{theorem}}

	
	The answers to Problem' for unbounded sets $K$ (except for a union of one or two unbounded intervals and a point) are given by the following result. 

\renewcommand{\thetheorem}{D}\begin{theorem} 
	Let $K$ be an unbounded closed semialgebraic set.
	
	The $n$-th matrix quadratic module $M_S^n$ is saturated for the natural description $S$ of $K$ 
	and every $n\in \NN$ if $K$ is either of the following:
				\begin{enumerate}
						\item An unbounded interval (by Theorem \ref{zvezni-fejer-riesz} and 
							\cite[Theorem 8]{Sch-Sav}).
						\item A union of two unbounded intervals (see Proposition \ref{dva intervala}).
				\end{enumerate}
			
	The $n$-th matrix preordering $T^{n}_S$ is not saturated for any finite set $S\subset \RR[x]
	$ such that
	$K=K_S$ in the following cases (see Theorem \ref{protiprimeri-za-sibko-omejenostno-nasicenost-R}):
		\begin{enumerate}
			\item $n\geq 2$ and $K$ contains at least two intervals with at least one of them bounded.
			\item $n\geq 2$ and $K$ is a union of an unbounded interval
				and $m$ isolated points with $m\geq 2$.
			\item	$n\geq 2$ and $K$ is a union of two unbounded intervals and $m$
				isolated points with $m\geq 2$.
		\end{enumerate}
\end{theorem}

\addtocounter{theorem}{-1}
\renewcommand{\thetheorem}{\arabic{section}.\arabic{theorem}}

	In the remaining cases of a union of one or two unbounded intervals and a point not covered by
Theorems C and D 
we state the following conjecture based on the investigation of some examples.

\begin{conjecture*} 
		Let $K\subseteq \RR$ be either of the following:
	\begin{enumerate}
		\item A union of an unbounded interval and a point.
		\item A union of two unbounded intervals and a point.
	\end{enumerate}
	Suppose $S$ is the natural description of $K$. Then the $n$-th matrix preordering 
	$T_S^n$ is saturated for every natural number $n>1$.
\end{conjecture*}

Note that by an appropriate substitution of variables both cases covered by Conjecture are equivalent. 


For the unbounded sets $K$ with a negative answer to Problem' we obtain the following characterization of the set $\Pos_{\succeq 0}^n(K)$.

\renewcommand{\thetheorem}{E}
\begin{theorem}
		Let $K$ be an unbounded closed semialgebraic set with a natural description $S$ and 
	$n\in \NN$.	Then the following statements are equivalent: 
		\begin{enumerate}
			\item $F\in \Pos_{\succeq 0}^n(K)$.
			\item For every $w\in \CC$ there exists $h\in\RR[x]$ such that 
				$h(w)\neq 0$ and $h^2 F\in T^n_S$ (see Theorem \ref{glavni}).
			\item For every $w\in \CC\setminus K$ there exists $k_w\in \NN\cup\{0\}$ such 
				that $$((x-\overline{w})(x-w))^{k_w} F\in T^n_S$$  (see Corollary \ref{glavni-2} and Remark 
				\ref{preostali}).
			\item $(1+x^2)^{k} F\in T^n_S$ for some $k\in \NN\cup\{0\}$ (Take $w=i$ in 3.).
		\end{enumerate}
\end{theorem}

\addtocounter{theorem}{-1}
\renewcommand{\thetheorem}{\arabic{section}.\arabic{theorem}}

	The following table summarizes \cite[Theorem 2.2]{K-M}, Theorems C, D and Conjecture.

$$\begin{array}{|c|c|c|c|} 
\hline
K &A & B \\
\hline
\text{a bounded set}  & \text{Yes} & \text{Yes}  \\
\hline
\text{an unbounded interval} & \text{Yes} & \text{Yes}   \\ 
\hline
\text{a union of an unbounded interval and an isolated point} & \text{Yes} & \text{C}   \\
\hline
\begin{tabular}{cc}
    \text{a union of an unbounded interval and}\\
    \text{$m$ isolated points with $m\geq 2$}
\end{tabular} 
& \text{Yes} & \text{No} \\
\hline
\text{a union of two unbounded intervals} & \text{Yes} & \text{Yes}   \\
\hline
\text{a union of two unbounded intervals and an isolated point} & \text{Yes} & \text{C}   \\
\hline
\begin{tabular}{cc}
    \text{a union of two unbounded intervals and}\\
    \text{$m$ isolated points with $m\geq 2$}
\end{tabular} 
& 
\text{Yes} & \text{No}  \\
\hline
\text{includes a bounded and an unbounded interval} & \text{Yes} & \text{No}\\
\hline
\end{array}$$
\begin{eqnarray*}
	A &:=& \text{ The } \text{preordering }T_S^1 \text{ is saturated for the natural description
			} S \text{ of }\\
		&& K.\\
	B &:=& \text{ The } n\text{-th matrix preordering }T_S^n \text{ is saturated 
		for the natural} \\ 	
		&&\text{description } S \text{ of } K\text{ and every integer } n\in\NN.\\
	C &:=& \text{ See Conjecture.}
\end{eqnarray*}

\begin{remark}
	\begin{enumerate}
	\item Since $T^1_S$ is saturated for the natural description $S$ of $K$, it follows that if
	$T^n_S$ is not saturated for some $n\in \NN$, then $T^n_{S_1}$ is not saturated for
	any finite set $S_1$ satisfying $K_{S_1}=K$.
	\item The classification covers all closed semialgebraic sets $K\subseteq \RR$. 	
A set $K$ is \textsl{regular} if it is equal to the closure of its interior. 
For regular sets $K\subseteq \RR$ the classification is complete.
	\end{enumerate}
\end{remark}


%


\section{Saturated descriptions of a compact set $K\subset \RR$ generate saturated $n$-th matrix quadratic modules}


The solution to Problem' from the Introduction for a compact set $K$ is the main result of this section (see Theorem \ref{nasicenost-na-R} below). It also characterizes all finite sets $S$ such that the quadratic module $M^n_S$ is saturated for every natural number $n\in \NN$.


\begin{theorem} \label{nasicenost-na-R}
	Suppose $K$ is a non-empty compact semialgebraic set in $\RR$.
	The $n$-th matrix quadratic module $M_S^n$ is saturated for every $n\in \NN$ iff $S$ a 
	saturated description of $K$. 
\end{theorem}

The main ingredients in the proof of Theorem \ref{nasicenost-na-R} are:
\begin{enumerate}
	\item	The $n=1$ case \cite[Theorem 5.17]{Scheiderer3}.
	\item The ``$h^2F$-proposition'' (See Proposition \ref{pomozna} below. 
		The proof uses the idea of diagonalizing 
		matrix polynomials from \cite[4.3]{Sch}.).
	\item Getting rid of $h^2$ in ``$h^2F$-proposition'' (The proof uses 
		\cite[Proposition 2.7]{Scheiderer}, which is Proposition \ref{basic-lemma-2} below.).
\end{enumerate}

\subsection{``$h^2F$-proposition''}

We call the following result $``h^2F$-proposition''.

\begin{proposition} \label{pomozna}
	Suppose $K$ is a non-empty compact semialgebraic set in $\RR$ with a saturated description $S$. 
	Then, for any $F\in \HH_n(\CC[x])$ such that $F\succeq 0$ on $K$ and every point $x_0
	\in \CC$, there exists $h\in \RR[x]$ such that $h(x_0)\ne 0$ and $h^2F\in M^n_S$. 
\end{proposition}

To prove Proposition \ref{pomozna}
we need Lemmas \ref{permutacijska-lema} and \ref{h-2-1-lema} below.

\begin{lemma} \label{permutacijska-lema}
	Let $G=[g_{kl}]_{kl}\in M_n\left(\CC\left[x\right]\right)$.
	For every $1\leq k\leq l\leq n$	there exist unitary matrices 
	$U_{kl}\in M_n(\RR)$ and $V_{kl}\in M_n(\CC)$ such that
		\begin{eqnarray*}
			U_{kl} G U_{kl}^\ast &=&
				\left[\begin{array}{cc} p_{kl} & \ast \\ \ast & \ast\end{array}\right],\quad
			V_{kl} G V_{kl}^\ast = 
				\left[\begin{array}{cc} r_{kl} & \ast \\ \ast & \ast\end{array}\right],
		\end{eqnarray*}
	where 
		\begin{eqnarray*}
			p_{kl}&=&
			\left\{\begin{array}{cc} g_{kl},& \text{for } 1\leq k=l\leq n\\
				\frac{1}{2}(g_{kl}+ g_{lk} + g_{kk} + g_{ll}),& \text{for } 1\leq k<l\leq n
				\end{array}\right.,\\
			r_{kl}&=& 
			\left\{\begin{array}{cc} g_{kl},& \text{for } 1\leq k=l\leq n\\
				\frac{i}{2}(-g_{kl}+ g_{lk}) + \frac{1}{2}(g_{kk} + g_{ll}),& \text{for } 1\leq k<l\leq n
				\end{array}\right..
		\end{eqnarray*}
\end{lemma}

\begin{proof}
	We define $U_{11}=V_{11}:=I_n$, $U_{kk}=V_{kk}:=P_k$ for $k = 2, \ldots, n$,
	where $P_k$ denotes the permutation matrix which permutes the first row and the $k$-th row.
	
	For $1\leq k<l\leq n$, define $U_{kl}:=P_k S_{kl}$ where 
	$S_{kl}=\left(s_{pr}^{(kl)}\right)_{pr}\in M_{n}(\RR)$
	is the matrix with 
		$s^{(kl)}_{kk} = s^{(kl)}_{kl} = s^{(kl)}_{lk} =\frac{1}{\sqrt{2}}$, 
		$s^{(kl)}_{ll} = -\frac{1}{\sqrt{2}}$,	
		$s^{(kl)}_{pp} = 1$ if $p \notin \left\{k, l\right\}$ and 
		$s^{(kl)}_{pr} =0$ otherwise.
	
	For $1\leq k<l\leq n$, define $V_{kl}:=P_k \tilde S_{kl}$ where 
	$\tilde S_{kl}=\left(\tilde s_{pr}^{(kl)}\right)_{pr}\in M_{n}(\CC)$
	is the matrix with 
		$\tilde s^{(kl)}_{kk} = \tilde s^{(kl)}_{lk} =\frac{1}{\sqrt{2}}$, 
		$\tilde s^{(kl)}_{kl} = \frac{i}{\sqrt{2}}$,	
		$\tilde s^{(kl)}_{ll} = -\frac{i}{\sqrt{2}}$,
		$\tilde s^{(kl)}_{pp} = 1$ if $p \notin \left\{k, l\right\}$
		and $\tilde s^{(kl)}_{pr}=0$ otherwise.	
\end{proof}

\begin{lemma} \label{h-2-1-lema}
	For $F=\left[\begin{array}{cc} a & \beta \\ \beta^{\ast} & \fC\end{array}\right]\in
	\HH_n\left(\CC\left[x\right]\right)$ where 
	$a=a^\ast\in \RR\left[x\right]$,
	$\beta \in M_{1,n-1}\left(\CC\left[x\right]\right)$ and 
	$\fC\in \HH_{n-1}\left(\CC\left[x\right]\right)$ it holds that
		\begin{eqnarray*}
		&\text{(i)}& a^4 \cdot F=
		\left[\begin{array}{cc} a^\ast & 0 \\ \beta^{\ast} & a^\ast I_{n-1}\end{array}\right]
		\left[\begin{array}{cc} a^3 & 0 \\ 0 & a(a\fC- \beta^\ast \beta)\end{array}\right]
		\left[\begin{array}{cc} a & \beta \\ 0 & a I_{n-1}\end{array}\right].\\
		&\text{(ii)}&
		\left[\begin{array}{cc} a^3 & 0 \\ 0 & a(a\fC- \beta^\ast \beta)\end{array}\right]=
		\left[\begin{array}{cc} a^\ast & 0 \\ -\beta^{\ast} & a^\ast I_{n-1}\end{array}\right]
		\cdot F\cdot
		\left[\begin{array}{cc} a & -\beta \\ 0 & a I_{n-1}\end{array}\right].
		\end{eqnarray*}
\end{lemma}

\begin{proof}
	Easy computation.
\end{proof}

\begin{proof}[Proof of Proposition \ref{pomozna}]
	The proof is by induction on the size $n$ of the matrix polynomials. For $n=1$ the proposition
	holds by the scalar case (We take $h=1$ and use \cite[Theorem 5.17]{Scheiderer3} and
	\cite[Corollary 4.4]{Scheiderer4}.). Suppose the proposition
	holds for $n-1$. We will prove that it holds for $n$.
	Let us take $F:=[f_{kl}]_{kl}\in \HH_n(\CC[x])$ where $F\succeq 0$ on $K$.
	Let us define
		$$c(x):=\left\{ 
			\begin{array}{rl} 
				x-x_0,& x_0\in  \RR \\ 
				(x-x_0)(x-\overline{x_0}), & 	x_0\in \CC\setminus \RR			
			\end{array} 
						\right..$$
	If $F\equiv 0$, we can take $h=1$. 
	Otherwise $F\not\equiv 0$ and we write 
		$$F=c^m G,$$ 
	where $m\in \NN\cup \{0\}$, $G=[g_{kl}]_{kl}\in \HH_n\left(\CC\left[x\right]\right)$
	and 
		\begin{equation}\label{G} G(x_0)=[g_{kl}(x_0)]_{kl}\neq 0. \end{equation}	
	
\noindent \textbf{Claim.} One of the following two cases applies:
\begin{description}[0.5cm]
		\item[Case 1:] 
			$g_{k_0k_0}(x_0)\neq 0 \text{ for some } k_0\in \{1,\ldots,n\}.$
		\item[Case 2:] $g_{kk}(x_0)= 0$ for all $k\in \{1,\ldots,n\}$ 
			and	for some $1\leq k_0<l_0\leq n$ we have
				$$\Re(g_{k_0 l_0})(x_0)\neq 0\quad \text{or}\quad \Im(g_{k_0 l_0})(x_0)\neq 0,$$ 
			where $\Re(g_{k_0 l_0}):=\frac{g_{k_0 l_0}+\overline{g_{k_0 l_0}}}{2}\in \RR[x]$ and
			$\Im(g_{k_0 l_0}):=\frac{g_{k_0 l_0}-\overline{g_{k_0 l_0}}}{2i}\in \RR[x]$.
\end{description}
	
\noindent	\textsl{Proof of Claim.} Let us assume that none of the two cases applies. Then $\Re(g_{kl})(x_0)=
	\Im(g_{kl})(x_0)= 0$ for all $1\leq k\leq l\leq n$. Let us take $l<k$. Since $G\in \HH_n\left(\CC\left[x\right]\right)$ is hermitian, it follows 
	that $g_{lk} = \overline{g_{kl}}=\Re g_{kl}-i\cdot\Im g_{kl}$.
	Therefore	$g_{lk}(x_0)	= \Re{g_{kl}}(x_0)-i\cdot \Im{g_{kl}}(x_0)=0.$
	Hence $g_{kl}(x_0)=0$ for all $k,l\in \{1,\ldots,n\}$. This is a contradiction with (\ref{G}) and proves Claim.\\
	
	Let $U_{kl}$, $V_{kl}$, $p_{kl}$, $r_{kl}$ be as in Lemma \ref{permutacijska-lema}.
	We study each case from Claim separately:\\
		
		\noindent \textbf{Case 1:} We define $T_{k_0k_0}:=U_{k_0k_0}$, $\tilde{g}_{k_0k_0}:=g_{k_0k_0}.$
		Notice that $\tilde{g}_{k_0k_0}(x_0)=g_{k_0k_0}(x_0)\neq 0$.\\
	
		\noindent \textbf{Case 2:}	We will separate three subcases:
			
		\indent \textsl{Subcase 2.1.} $p_{k_0l_0}(x_0)\neq 0$: We define
				$T_{k_0l_0}:=U_{k_0l_0}, \tilde{g}_{k_0l_0}:=p_{k_0l_0}.$
		Notice that $\tilde{g}_{k_0l_0}(x_0)\neq 0$.
		
		\indent \textsl{Subcase 2.2.} $r_{k_0l_0}(x_0)\neq 0$: We define
				$T_{k_0l_0}:=V_{k_0l_0}, \tilde{g}_{k_0l_0}:=r_{k_0l_0}.$
		Notice that $\tilde{g}_{k_0l_0}(x_0)\neq 0$.
		
		\indent \textsl{Subcase 2.3.} $p_{k_0l_0}(x_0)=r_{k_0l_0}(x_0)=0$:
			We will prove that this subcase does not happen. By definition and assumptions we have
				\begin{eqnarray*} 
					p_{k_0l_0}(x_0)&=& \frac{1}{2}(g_{k_0l_0}+ g_{l_0k_0} + g_{k_0k_0} + g_{l_0l_0})(x_0)=
						\frac{1}{2}(g_{k_0l_0}+ g_{l_0k_0})(x_0)=\\
							&=&(\Re g_{k_0l_0})(x_0)\\
					r_{k_0l_0}(x_0)&=& \frac{i}{2}(-g_{k_0l_0}+ g_{l_0k_0})(x_0) + \frac{1}{2}(g_{k_0k_0} + g_{l_0l_0})(x_0)=\\
						&=&\frac{i}{2}(-g_{k_0l_0}+ g_{l_0k_0})(x_0)=(\Im g_{k_0l_0})(x_0)
				\end{eqnarray*} 
			Since we are in Case 2, $(\Re g_{k_0l_0})(x_0)\neq 0$ or $(\Im g_{k_0l_0})(x_0)\neq 0$.
			Contradiction. Hence Subcase 2.3 never happens.\\
			
	To avoid repetition in what follows we define $k_0=l_0$ if we are in Case 1.
	If we write
	$T_{k_0l_0} G T_{k_0l_0}^\ast=\left[\begin{array}{cc} \tilde{g}_{k_0l_0} & \tilde \beta \\
	\tilde \beta^\ast & \tilde \fC\end{array}\right]$
	with $\tilde \beta\in M_{1,n-1}\left(\CC\left[x\right]\right)$ and 
	$\tilde \fC\in M_{n-1}\left(\CC\left[x\right]\right)$, then
	$T_{k_0l_0} F T_{k_0l_0}^\ast
			=
			\left[\begin{array}{cc} c^m\tilde g_{k_0l_0} & \fp^m\tilde \beta \\
			(\fp^m\tilde \beta)^\ast & \fp^m \tilde \fC\end{array}\right]=:
			\left[\begin{array}{cc} a & \beta \\ \beta^{\ast} & \fC\end{array}\right].$
	Therefore by part $(i)$ of Lemma \ref{h-2-1-lema} and dividing by $c^{4m},$
	it follows that 
		\begin{eqnarray*}
			\tilde g^2  F &=& 
			T_{k_0l_0}^\ast\left[\begin{array}{cc} 
		\tilde{g}_{k_0l_0}^\ast & 0\\
		\tilde\beta^\ast & 
		\tilde{g}_{k_0l_0}^\ast I_{n-1}
		\end{array}\right]
		\left[\begin{array}{cc} 
		d & 0\\
		0 & D
		\end{array}\right]
		\left[\begin{array}{cc} 
		\tilde{g}_{k_0l_0} & \tilde\beta\\
		0 & \tilde{g}_{k_0l_0}I_{n-1}
		\end{array}\right] T_{k_0l_0},
		\end{eqnarray*}
	where 
		\begin{eqnarray*}
			\tilde g &=& \tilde g_{k_0l_0}^2 
				\in \HH_1\left(\CC\left[x\right]\right)=\RR[x]\\
			d &=&	c^m\tilde{g}_{k_0l_0}^3\in \HH_1\left(\CC\left[x\right]\right)=\RR[x],\\
			D &=&	c^m\tilde{g}_{k_0l_0}\left(\tilde{g}_{k_0l_0}\tilde \fC 
							-\tilde \beta^\ast\tilde\beta\right)\in \HH_{n-1}\left(\CC\left[x\right]\right).
		\end{eqnarray*}
	By part $(ii)$ of Lemma \ref{h-2-1-lema} and dividing by $c^{2m}$, we 
	have also
	\begin{eqnarray*}
		\left[\begin{array}{cc} 
		d & 0\\
		0 & D
		\end{array}\right] &=&
		\left[\begin{array}{cc} 
		\tilde{g}_{k_0l_0}^\ast & 0\\
		-\tilde\beta^\ast & 
		\tilde{g}_{k_0l_0}^\ast I_{n-1}
		\end{array}\right] T_{k_0 l_0}
		F T_{k_0l_0}^\ast
		\left[\begin{array}{cc} 
		\tilde{g}_{k_0l_0} & -\tilde\beta\\
		0 & \tilde{g}_{k_0l_0} I_{n-1}
		\end{array}\right]
	\end{eqnarray*}
		It follows that $\fd\geq 0$, $\fD\succeq 0$ on $K$. By the induction hypothesis used for 
		the polynomial $\fD\in \HH_{n-1}\left(\CC\left[x\right]\right)$,
	there exists $h_1\in \RR\left[x\right]$ such that 
	$h_1(x_0)\neq 0$ and 
	$h_1^2 D\in M^{n-1}_{S}$. By the scalar case 
	\cite[Theorem 5.17]{Scheiderer3} and \cite[Corollary 4.4]{Scheiderer4},  
	$h_1^2 d\in M^{1}_{S}$.
	Hence
	$h^2 F \in M^n_{S}$
	where $h=h_1 \tilde g\in \RR[x]$ and $h(x_0)\neq 0$. This concludes the proof.
\end{proof}

\begin{remark} \label{opomba-o-stopnji}
	By keeping track on the degree of $h$ and using \cite[Theorem 4.1]{K-M-2}, we can prove more in Proposion \ref{pomozna} above. Namely,
	$h$ can be chosen of degree at most $\text{deg}(F)(3^n-1)$ and if $S=\{g_1,\ldots,g_s\}$ is the natural description of $K$, then 
	$F=\sum_{e\in \{0,1\}^s}\sigma_e \underline{g}^e\in T^n_S$ for some $\sigma_e\in M_n(\CC[x])^2$ with $\text{deg}(\sigma_e \underline{g}^e)\leq \text{deg}(h^2F).$
\end{remark}

	\subsection{ Getting rid of $h^2$ in ``$h^2F$-proposition''}
	
	To get rid of $h^2$ in ``$h^2F$-proposition'', which proves Theorem \ref{nasicenost-na-R}, we will use \cite[Proposition 2.7]{Scheiderer}: 
	
	\begin{proposition} \label{basic-lemma-2}
		Suppose $R$ is a commutative ring with $1$ and $\QQ\subseteq R$. 
		Let	$\Phi:R\to C(K,\RR)$ be a ring homomorphism, where $K$ is a topological space 
		which is compact and Hausdorff. Suppose $\Phi(R)$ separates points in $K$. Suppose 
		$f_1,\ldots, f_k\in R$ are such that $\Phi(f_j)\geq 0$, $j=1,\ldots, k$ and
		$(f_1,\ldots,f_k)=(1)$. Then there exist $s_1,\ldots, s_k\in R$ such that $s_1f_1+\ldots+s_k f_k=1$
		and such that each $\Phi(s_j)$ is strictly positive.
	\end{proposition}
	
	\begin{proof}[Proof of Theorem \ref{nasicenost-na-R}]
		By \cite[Corollary 4.4]{Scheiderer4} and \cite[Theorem 5.17]{Scheiderer3},
		$M_S^1$ is saturated if and only if $S$ is a saturated description of $K$.
		Therefore we have to prove only the if part. Let
		$S$ be a saturated description of $K$. 
		We will prove that $M^n_S$ is saturated for every $n\in \NN$.
		Let $R:=\RR[x]$ and $\Phi:R\rightarrow C(K,\RR)$ be the natural map, i.e.,\ $\Phi(f)=f|_K$.
		Take $F\in \Pos_{\succeq 0}^n(K)$. We will prove that $F\in M^n_S$. Let
			$I:=\left\langle h^2\in \RR[x]\colon h^2 F\in M^n_S\right\rangle$
		be the ideal in $\RR[x]$ generated by all $h^2$ where $h\in \RR[x]$ is such that $h^2F\in M^n_S$.
		Since $\RR[x]$ is a principal ideal domain, there exists a polynomial
		$p\in \RR[x]$ such that $I=\left\langle p \right\rangle$. If $I$ was a proper ideal, all its elements would have a common zero $x_0\in \CC$.
		By Proposition \ref{pomozna}, there exists $h\in \RR[x]$ such that
		$h(x_0)\neq 0$ and $h^2 F\in M^n_S$. Since $h$ belongs to $I$, it follows that 
		$I$ is not a proper ideal and hence $I=\RR[x]$.
		By Proposition \ref{basic-lemma-2}, there exist $s_1,\ldots,s_k \in \Pos_{\succ 0}^1(K)$ and
		$h_1,\ldots,h_k \in I$ such that
		$\sum_{j=1}^k s_jh_j^2=1$. Hence $\sum_{j=1}^k s_jh_j^2 F=F\in M^n_S$,
		which concludes the proof.
	\end{proof}

\begin{remark}	\label{drugi-dokaz}
\begin{enumerate}
	\item
	There is another proof of Theorem \ref{nasicenost-na-R} which uses Proposition \ref{pomozna} 
	just for the boundary points of $K$. We outline the main idea.
	There exists $h\in \RR[x]$ such that 
	$h\in \Pos_{\succeq 0}^1(\RR)$, $h(x_0)>0$ for every boundary point of $K$ and $hF\in M_S^n$ 
	(Take $h=\sum_{x_0\in \partial K} h_{x_0}^2$ where $\partial K$ is the boundary of $K$ and 
	$h_{x_0}$ is the polynomial from Proposition \ref{pomozna} for the point $x_0$.).
	Now multiply every member of the set $S$ by $h$ to obtain the set $S_1$ which satisfies conditions of 
	\cite[Corollary 5.17]{Scheiderer3}. Thus $M^1_S=M^1_{S_1}$ and $hF\in M_{S_1}^n$. This means there 
	exist $\sigma_j\in \sum M_n(\CC[x])^2$ such that $hF=\sigma_0+ \sigma_1 hg_1+\ldots+\sigma_s hg_s.$
	From here it is easy to see that $F=\tau_0+\sigma_1 g_1+\ldots+\sigma_s g_s$ for some 
	$\tau_0\in \sum M_n(\CC[x])^2$ and hence $F\in M_S^n$.
	\item By Remark \ref{opomba-o-stopnji}, the degree of $h$ in Proposition \ref{pomozna} and the degrees of summands in the 
	expression of $h^2F$ as the element of the preordering $T^n_S$ generated by the natural description $S$ of $K$ can be bounded by the degree of $F$ and $n$. It would be interesting to know if
	the same holds for $F$ and an arbitrary compact set $K$. It can be shown this is true for a finite set $K$. The degrees can be bounded by $\max(\text{deg}(F),|K|-1)$.
\end{enumerate}
\end{remark}

	\section{Unbounded sets $K$ without saturated $T_S^2$ for any finite sets $S$ with $K_S=K$}
	
	The answer to the question of Problem' for unbounded sets $K$ is positive for an unbounded interval by
	Theorem \ref{zvezni-fejer-riesz} (if $K=\RR$) and	\cite[Theorem 8]{Sch-Sav} (if $K=[a,\infty)$). 
	It is also easy to derive a positive answer for a union of two unbounded intervals from the case 
	$K=[a,b]$:
	
	\begin{proposition}\label{dva intervala}
		Let $K=(-\infty,a]\cup [b,\infty)$ be a union of two unbounded intervals where
		$a,b\in \RR$ and $a<b$.
		Then the quadratic module $M^n_{\{(x-a)(x-b)\}}$ is saturated for every $n\in \NN$.
	\end{proposition}
	
	\begin{proof}
				By a linear change of variables, we may assume that $K=(-\infty,-1]\cup[1,\infty)$.
				Note that $F\in \Pos^{n}_{\succeq 0}(K)$ is of even degree. 
				We define 
					$$F_1(x)=x^{\deg(F)} F\left(\frac{1}{x}\right)$$ 
				and observe that 
				$F_1\succeq 0$ on $[-1,1]$. By \cite[Theorem 2.5]{ds} and by the identity
					$$1\pm x=\frac{(1\pm x)^2+(x+1)(1-x)}{2},$$
				there exist matrix polynomials
				$G_1$, $H_1$ such that
					$$F_1(x)=G_1(x)^\ast G_1(x)+H_1(x)^\ast H_1(x) (x+1)(1-x),$$
					$$\deg(G_1)\leq \left\lfloor \frac{\deg(F_1)}{2}\right\rfloor\leq \frac{\deg(F)}{2},$$
					$$\deg(H_1)\leq \left\lfloor \frac{\deg(F_1)-1}{2}\right\rfloor\leq
						\left\lfloor \frac{\deg(F)-1}{2}\right\rfloor=\frac{\deg(F)}{2}-1.$$
				Therefore
				\begin{eqnarray*}
					F(x) &=& x^{\deg(F)}F_1(\frac{1}{x})\\
							 &=&
										x^{\deg(F)}
										(G_1(\frac{1}{x})^\ast G_1(\frac{1}{x})
										+H_1(\frac{1}{x})^\ast H_1(\frac{1}{x})
										(\frac{1}{x}+1)(1-\frac{1}{x})
										)\\
				&=:& G(x)^\ast G(x)+ H(x)^\ast H(x) (1+x)(x-1),
	\end{eqnarray*}
		where 
			$$G(x):=x^{\frac{\deg(F)}{2}}G_1\left(\frac{1}{x}\right),\quad
			H:=x^{\frac{\deg(F)}{2}-1}H_1\left(\frac{1}{x}\right)$$ 
		are matrix polynomials. 
	\end{proof}
	
	The negative answer to the question of Problem' 
	for almost all remaining unbounded sets $K$ (except 
	for a union of an unbounded interval and a point 
	or a union of two unbounded intervals and a point) and all 
	$n\geq 2$ is the main result of this section.

	\begin{theorem} \label{protiprimeri-za-sibko-omejenostno-nasicenost-R}
		Let an unbounded closed semialgebraic set $K\subseteq \RR$ satisfy either of the 
		following:
		\begin{enumerate}
			\item $K$ contains at least two intervals with at least one of them bounded.
			\item $K$ is a union of an unbounded interval and $m$ isolated points with $m\geq 2$.
			\item $K$ is a union of two unbounded intervals and $m$ isolated 
				points with $m\geq 2$.
		\end{enumerate}
			If $S\subset \RR[x]$ is a finite set with $K_S=K$, then the $2$-nd matrix preordering $T_S^2$ is 
		not saturated.
	\end{theorem}
	
	It is sufficient to prove Theorem \ref{protiprimeri-za-sibko-omejenostno-nasicenost-R} for the natural description $S$ of $K$ by the following lemma.

\begin{lemma}\label{redukcija-na-naravno-deskripcijo}
	Let $K\subseteq \RR$ be an unbounded closed semialgebraic set
	with the natural description $S$. Let $S_1\subset \RR[x]$ be a finite set such that
	$K_{S_1}=K$. 
	For every $n\in \NN$ such that 
	the $n$-th matrix preordering $T^n_S$ is not saturated, also the $n$-th 
	matrix preordering $T^n_{S_1}$ is not saturated.
\end{lemma}
	
	\begin{proof}
		Let us write $S:=\left\{g_1,\ldots,g_s\right\}$ and 
		$S_1:=\left\{f_1,\ldots,f_t\right\}$.
		We have to show that every matrix polynomial $F$ from $T^n_{S_1}$ also belongs to $T^n_S$. 
		A matrix polynomial $F$ from $T^n_{S_1}$ is of the form 
			\begin{equation} \label{izraz-1}
				F=\sum_{e'\in \{0,1\}^t} \tau_{e'} 	f_1^{e'_1}\ldots f_t^{e'_t},
			\end{equation}
		where $e':=(e'_1,\ldots,e'_t)$ and $\tau_{e'}\in \sum M_n\left(\CC[x]\right)^2$.
		By \cite[Theorem 2.2]{K-M}, the preordering	$T^1_{S}$ is saturated and thus for each
		$j$ there exist $\sigma_{e,j}\in \sum \RR[x]^2$ such that
			\begin{equation}\label{izraz-2}
				f_j=\sum_{e\in \{0,1\}^s} \sigma_{e,j} \,g_1^{e_1}\cdots g_s^{e_s},
			\end{equation}
		where $e:=(e_1,\ldots,e_s)$.
		Plugging (\ref{izraz-2}) into (\ref{izraz-1}) and rearranging terms we obtain $F\in T^n_S$. 
		This concludes the proof.
	\end{proof}
	
	
	
	In the remaining part of this section we will prove Theorem \ref{protiprimeri-za-sibko-omejenostno-nasicenost-R}.	The major step will be Proposition \ref{trditev-protiprimer-interval-in-poltrak}.
	
	Let $K$ be a closed semialgebraic set with a natural description $S=\{g_1,\ldots,g_s\}$. For $
	n\in \NN$ and $d\in \NN\cup\{0\}$ we define the set 
	\begin{equation*} 
		T^n_{S,d}:=\left\{\sum_{e\in \{0,1\}^s} \sigma_e \underline{g}^e\colon 
		\sigma_e \in \sum M_n(\CC[x])^2\; \text{and} \;
		\deg(\sigma_e \underline{g}^e)\leq d \; \forall e \in \{0,1\}^s\right\}.
	\end{equation*}
	
	\begin{proposition}	\label{trditev-protiprimer-interval-in-poltrak}
		Let $K=[x_1,x_2]\cup [x_3,\infty)$ be a union of a bounded and an unbounded interval
		where $x_1<x_2<x_3$.
		Let us define the polynomial 
			$$F_k(x):=\left[\begin{array}{cc} x+A(k) & D(k) \\ D(k)& x^2+B(k)x+C(k)\end{array}\right],$$ 
		where
		\begin{eqnarray*}
			A(k) &:=& k-x_1,\\
			B(k) &:=& -k-x_2-x_3,\\
			C(k) &:=& k^2+k(-x_1+x_2+x_3)+x_2x_3,	\\
			D(k) &:=& \sqrt{A(k)C(k)+x_1x_2x_3}=\\
					 &=& \sqrt{k^3+k^2(-2x_1+x_2+x_3)+k(x_2x_3+x_1^2-x_1x_2-x_1x_3)}. 
		\end{eqnarray*}
		We define $p_k(x):=x^2+B(k)x+C(k)$.
		For every $k\in \RR$ which satisfies			
			\begin{equation}\label{pogoj 3}
				k>0,
			\end{equation}
			\begin{equation}\label{pogoj 1}
				D(k)^2=k^3+k^2(-2x_1+x_2+x_3)+k(x_2x_3+x_1^2-x_1x_2-x_1x_3)>0,
			\end{equation}
			\begin{equation}\label{pogoj 2}
				p_k\left(-\frac{B(k)}{2}\right)=\frac{3}{4}k^2+k\left(-x_1+\frac{x_2+x_3}{2}\right)-
			\left(\frac{x_2-x_3}{2}\right)^2>0,
			\end{equation}
		the matrix polynomials $F_k(x)$	belongs to $\Pos_{\succeq 0}^2(K)$, but:
		\begin{description}
			\item[Claim 1.] 
				$F_k\notin T^2_{S_1}$ where $S_1$ is the natural description of any set $K_1$ of the
				form 
					$$[x_1,x_2]\cup \cup_{j=1}^{m} [x_{2j+1},x_{2j+2}] \cup [x_{2m+3},{\infty})\subseteq K$$ 
				with $m\in \NN\cup\{0\}$ and $x_j\leq x_{j+1}$ for each $j$ (and $x_1<x_2<x_3$). In particular, 
					$$F_k(x)\notin T^2_S,$$
				where $S$ is the natural description of $K$.
			\item[Claim 2.] 
				$F_k\notin T^2_{S_2,2}$ where $S_2$ is the natural description of any set $K_2$ of the
				form 
					$$[x_1,x_2]\cup \cup_{j=3}^{m}\{x_{j}\}\subset K$$ 
				with $m\in \NN$, $m\geq 4$ and $x_j< x_{j+1}$ for each $j$.
		\end{description} 
	\end{proposition}
	
	\begin{proof}
		First we will prove that $F_k(x)$ belongs to $\Pos_{\succeq 0}^2(K)$ for every $k\in \RR$ 
		satisfying the conditions $(\ref{pogoj 3})$-$(\ref{pogoj 2})$. Note that every sufficiently large 
		$k$ satisfies the conditions $(\ref{pogoj 3})$-$(\ref{pogoj 2})$. 
		Condition $(\ref{pogoj 1})$ ensures that $D(k)\in \RR$ and hence 
		$F\in \HH_n(\RR[x])$.
		The determinant of $F_k(x)$ is $(x-x_{1})(x-x_{2})(x-x_{3})\in \Pos^1_{\succeq 0}(K).$
		The upper left corner of $F$ is non-negative for $x\geq x_1-k$ and hence it belongs to 
		$\Pos^1_{\succeq 0}(K)$ by (\ref{pogoj 3}). The lower right corner is a quadratic polynomial
		$p_k(x)$ with a vertex in $x=\frac{-B(k)}{2}$. Since $k$ satisfies
		(\ref{pogoj 2}), 
			$p_k\left(\frac{-B(k)}{2}\right)>0.$
		So $p_k(x)$ is positive on $\RR$ and hence $p_k\in\Pos^1_{\succeq 0}(K)$. 
		Since all principal minors of $F_k(x)$ are non-negative on $K$, the conclusion
		$F_k(x)\in \Pos^2_{\succeq 0}(K)$ follows.
		
		We will separately prove both claims of the theorem.\\
		
	\noindent\textbf{Proof of Claim 1.}
		The set
			$$\{\underbrace{x-x_1}_{g_1(x)},\underbrace{(x-x_2)(x-x_3)}_{g_2(x)},\ldots, 
				\underbrace{(x-x_{2m+2})(x-x_{2m+3})}_{g_{m+2}(x)}\}$$
	is the natural description $S_1$ of $K_1$. We will prove that $F_k(x)\notin T^2_{S_1}$  by 
	contradiction. 
	Let us assume $F_k\in T_{S_1}^2$.
	Then for every
	$e:=(e_1,\ldots,e_{m+2})\in \{0,1\}^{m+2}$ there exists $\sigma_e\in \sum M_n(\CC[x])^2$, such that
		\begin{equation}\label{enakost 2}
			F_k=\sum_{e\in \{0,1\}^{m+2}} \sigma_e g_1^{e_1}\cdots g_{m+2}^{e_{m+2}}.
		\end{equation}
	By the degree comparison of both sides of (\ref{enakost 2}), there exist $\sigma_j\in
	\sum M_n(\CC[x])^2$, such that
		\begin{equation}\label{enakost 1}
			F_k(x)=\sigma_0 + \sigma_1 (x-x_{1}) + \sum_{j=1}^{m+1} \sigma_{j+1} (x-x_{2j})(x-x_{2j+1}),
		\end{equation}
		$$\deg(\sigma_0)\leq 2, \quad \deg(\sigma_j)=0 \text{ for } j=1,\ldots,m+2.$$
	By observing the monomial $x^2$ on both sides of (\ref{enakost 1}), it follows that
	$\sigma_2=\left[\begin{array}{cc} 0 & 0 \\ 0 & k_0 \end{array}\right]$ for some $k_0\in [0,1]$.
	Equivalently, (\ref{enakost 1}) can be written as 
		$$F_k(x)-\sigma_2 (x-x_2)(x-x_3) = \sigma_0+\sigma_1(x-x_1)
			+\sum_{j=2}^{m+1} \sigma_{j+1} (x-x_{2j})(x-x_{2j+1}).$$
	The right-hand side belongs to	
	$\Pos^2_{\succeq 0}(\hat K_1)$ where
	$\hat{K}_1=K_1\cup [x_2,x_3]$. We will prove that the left-hand side does not belong to
	$\Pos^2_{\succeq 0}(\hat K_1)$, which is a contradiction.
	The determinant of the left-hand side is
		$$q(x):=(x-x_{2})(x-x_{3})(x(1-k_0)-(x_{1}-x_{1}k_0+kk_0)).$$ 
	There are two cases two consider: $k_0=0$ and $k_0>0$. In the first case, $q(x)=(x-x_1)(x-x_2)(x-x_3)$
	which is negative on $(x_2,x_3)$, a contradiction with $q|_{\hat K_1}\geq 0$. In the second case,
	$q(x_1)=(x_1-x_2)(x_1-x_3)(-kk_0)<0,$ which is also a contradiction with $q|_{\hat K_1}\geq 0$.
	Thus 
		$$F_k(x)-\sigma_2 (x-x_{2})(x-x_{3})\notin \Pos_{\succeq 0}^2(\hat K_1),$$ 
	which is a contradiction.
	Therefore $F_k$ cannot be expressed in the form (\ref{enakost 2}) and so $F_k\notin T^2_{S_1}$.\\

	\noindent\textbf{Proof of Claim 2.} The set
			$$\{\underbrace{x-x_1}_{g_1(x)},\underbrace{(x-x_2)(x-x_3)}_{g_2(x)},\ldots, 
				\underbrace{(x-x_{m-1})(x-x_{m})}_{g_{m-1}(x)},\underbrace{x_m-x}_{g_{m}(x)}\}$$
	is the natural description $S_2$ of $K_2$. If $F_k\in T^2_{S_2,2}$, then
	there exist $\tau_j\in \sum M_n(\CC[x])^2$ such that
	\begin{equation}\label{enakost 5}
			F_k(x)=\tau_0 + \tau_1 (x-x_{1}) + \sum_{j=2}^{m-1} \tau_{j} (x-x_{j})(x-x_{j+1}) + 
			\tau_{m}(x_m-x)+ \tau_{m+1}(x-x_{1})(x_m-x),
	\end{equation}
$$\deg(\tau_0)\leq 2,\; \deg(\tau_j)=0 \text{ for } j=1,\ldots,m+1.$$
	From (\ref{enakost 5}) it follows that 
	\begin{equation} \label{enakost 8}
		(F_k(x)-\tau_{j}(x-x_j)(x-x_{j+1}))|_{K_2}\succeq 0 \quad \text{for } j=2,\ldots,m-1.
	\end{equation}
	From (\ref{enakost 8}) it follows that
		$$\ker F_k(x_1)\subseteq \ker\tau_{j},\; \ker F_k(x_2)\subseteq \ker\tau_{j} 
			\quad \text{for }	j=3,\ldots,m-1.$$
	Since $\ker F_k(x_1)\oplus \ker F_k(x_2)=\CC^2$, we conclude that $\tau_{j}=0$ for $j=3,\ldots,m-1$.
	
	Hence (\ref{enakost 5}) becomes
	\begin{equation*}
			F_k(x)=\tau_0 + \tau_1 (x-x_{1}) + \tau_{2} (x-x_{2})(x-x_{3}) +\tau_{m}(x_m-x)+
			\tau_{m+1}(x-x_{1})(x_m-x),
	\end{equation*}
	or equivalently,
	\begin{equation}\label{enakost 6}
			F_k(x)-\tau_{2} (x-x_{2})(x-x_{3}) =\tau_0 + \tau_1 (x-x_{1}) +\tau_{m}(x_m-x)+
			\tau_{m+1}(x-x_{1})(x_m-x).
	\end{equation}
	Since the determinant of the left hand side is of degree 4 and is divisible by 
	$(x-x_1)(x-x_2)(x-x_3)$ (divisibility by $x-x_1$ is due to $\ker F_k(x_1)\neq \{0\}$ and 
	(\ref{enakost 8}) for $j=2$), 
	it cannot be non-negative on $[x_1,x_m]$ (This follows by a simple geometric argument.).
	Hence the left-hand side of (\ref{enakost 6}) does not belong to $\Pos_{\succeq 0}^2([x_1,x_m])$,
	while the right-hand side does. This is a contradiction and thus $F_k\notin T^2_{S_2,2}$.
	\end{proof}
	
		%
		%
		%
	
	\begin{proof}[Proof of Theorem \ref{protiprimeri-za-sibko-omejenostno-nasicenost-R}.1]	
		By Lemma \ref{redukcija-na-naravno-deskripcijo}, we may assume that $S$ is the
		natural description of $K$.
		Let us write $K$ in the form $K_0\cup K_1$ where $K_0$ is the set of isolated points of $K$ and
		$K_1$ is the regular part of $K$ (i.e., does not have isolated points). 
		We separate three cases depending on the form of $K_1$.\\
		
		\noindent\textbf{Case 1: }\textsl{$K_1$ is bounded from below and unbounded from above.}
			Let us divide the isolated part $K_0$ into disjoint sets $K_{01}$, $K_{02}$
			where in $K_{01}$ are all those points which are smaller than the minimum of $K_1$
			and in $K_{02}$ all the others. The set $K_2:=K_1\cup K_{02}$ is of the form
				\begin{eqnarray*}
					[x_1,x_2]\cup \cup_{j=1}^{p} [x_{2j+1},x_{2j+2}] \cup [x_{2p+3},{\infty}),
				\end{eqnarray*}
			where $p\in \NN\cup \{0\}$, $x_1<x_2<x_3$ and $x_j\leq x_{j+1}$ for each $j\geq 3$. 
			Let us take a polynomial $F_1\in \Pos_{\succeq 0}^{2}(K_2)$ and define the polynomial 
				\begin{equation}\label{enakost 13}
					F(x):=\prod_{y\in K_{01}}(x-y)\cdot F_1(x) \in\Pos_{\succeq 0}^2(K).
				\end{equation}
			Let $S:=\{g_1,\ldots,g_s\}$ be the natural description of $K$.
		If $F$ belongs to $T^2_{S}$, then for every $e\in \{0,1\}^{s}$ there exists $\sigma_e\in \sum M_n(\CC[x])^2$ 
		such that
			\begin{equation} \label{enakost 12} 
				F=\sum_{e\in \{0,1\}^{s}} \sigma_e \underline g^e.
			\end{equation}
		Since for every $y\in K_{01}$ and every $e\in \{0,1\}^{s}$ we have $F(y)=0$ and $\sigma_e
		\underline g^e(y)\succeq 0$,	
		it follows from (\ref{enakost 12}) that	$\sigma_e \underline g^e(y)=0$.
		Therefore $\prod_{y\in K_{01}}(x-y)$ divides each 
		$\sigma_e \underline g^e$. \\
		
		\noindent \textbf{Claim.} 
		There exist $\tau_e \in \sum M_n(\CC[x])^2$ and 
		$h_e\in \Pos^1_{\succeq 0}(K_2)$ such that
			$$\frac{\sigma_e \underline g^e}{\prod_{y\in K_{01}}(x-y)}=\tau_e h_e.$$
		
		\noindent \textsl{Proof of Claim.} Let us take $y\in K_{01}$. We separate two possibilities.
		\begin{enumerate}
			\item \textsl{$x-y$ divides $\sigma_e$:} 
				Then	
					$\sigma_e \underline g^e=\hat\sigma_e \cdot (x-y)^2\underline g^e$
				where 
					$\hat\sigma_e\in \sum M_n(\CC[x])^2$ 
				and 
					$\frac{(x-y)^2\underline g^e}{x-y}= (x-y)\underline g^e
				\in\Pos^1_{\succeq 0}(K_2)$. 
			\item \textsl{$x-y$ does not divide $\sigma_e$:} 
				Then $x-y$ divides $\underline g^e$ and hence
					$\sigma_e \underline g^e=\sigma_e \cdot (x-y)\hat g_e$
				where 
					$\hat g_e:=\frac{\underline g^e}{x-y} \in\Pos^1_{\succeq 0}(K_2)$.
		\end{enumerate}
		Repeating the above procedure 
		for every $y\in K_{01}$ we obtain $\tau_e$ and $h_e$ proving Claim.\\
		
		Let $S_2$ be the natural description of $K_2$.
		By \cite[Theorem 2.2]{K-M}, 
		$h_e\in T_{S_2}^1$. It follows that $F_1=\sum_{e}\tau_e h_e\in T_{S_2}^2$. 
		
		We have proved that for $F_1\in \Pos_{\succeq 0}^{2}(K_2)$ and $F\in\Pos_{\succeq 0}^{2}(K)$ defined
		by (\ref{enakost 13}),
		from $F\in T_{S}^2$ it follows that $F_1\in T_{S_2}^2$. Therefore, to find 
		$F\in\Pos_{\succeq 0}^{2}(K)$ and	$F\notin T_{S}^2$, it is sufficient to find $F_1\in \Pos_{\succeq 
		0}^{2}(K_2)$ and $F_1\notin T_{S_2}^2$.
		Let us define the set	$K_3 := [x_{1},x_{2}] \cup [x_3,\infty).$
		By Claim 1 of Proposition \ref{trditev-protiprimer-interval-in-poltrak}, there exists a polynomial 
		$F_1\in \Pos_{\succeq 0}^2(K_3)\subseteq \Pos_{\succeq 0}^2(K_2)$
		such that $F_1\notin T^2_{S_2}$. This proves Case 1.\\

		\noindent\textbf{Case 2:} \textsl{$K_1$ is unbounded from below and bounded from above.}
			Make a substitution $x\mapsto -x$ and observe that the set $-K_1$ is of the form in Case 1 and
			that the natural description of $K$ maps into the natural description of $-K$.\\
			%
			
		\noindent\textbf{Case 3:} \textsl{$K_1$ is unbounded from below and above.}
			Let $d\in \RR$ be the smallest endpoint of $K_1$.
			Define the map 
				$\lambda_d: \RR\setminus \{d\} \to \RR$ with $\lambda_d(x):=\frac{1}{d-x}.$
			Observe that $\lambda_d(K_1)=:K_2$ is the set of the form 
			$[x_1,x_{2}] \cup [x_3,x_4]\cup \ldots\cup [\hat x_{2m+1},\infty)$
			where $m\in \NN$ and $x_{j}<x_{j+1}$ for every $j$.
			Let $S_3$ be the natural description of $\lambda_d(K)$. 
			As in Case 1, construct the polynomial $F\in \Pos^{2}_{\succeq 0}(\lambda_d(K))$
			such that $F\notin T^2_{S_3}$.
			Now 
					$G(x)=x^{\left(2\left\lceil \frac{\deg(F)}{2}\right\rceil\right)}\cdot F\left(d-\frac{1}{x}\right)\in \Pos^2_{
						\succeq 0}(K)$ 
				and 
			$G\notin T^2_{S}$.
	\end{proof}
		
	\begin{proof} [Proof of Theorem \ref{protiprimeri-za-sibko-omejenostno-nasicenost-R}.2 and 
	\ref{protiprimeri-za-sibko-omejenostno-nasicenost-R}.3]
		By Lemma \ref{redukcija-na-naravno-deskripcijo}, we may assume that $S$ is the 
		natural description of $K$. Let $d\in \RR$ be an arbitrary point such that $d\notin K$.
		Define the map 
				$\lambda_d: \RR\setminus \{d\} \to \RR$ with $\lambda_d(x):=\frac{1}{d-x}.$
			Observe that $\lambda_d(K)$ is the set of the form
			$[x_1,x_2] \cup \cup_{j=3}^m\{x_j\}$
		where $m\geq 4$ and the points $x_j$ are pairwise different. Further on, we may choose $d\in \RR$
		such that
			$x_1<x_2<x_3<\ldots<x_m$ or $x_m<x_{m-1}<\ldots<x_3<x_1<x_2.$
		By substitution $x\mapsto -x$, we may assume that $x_1<x_2<x_3<\ldots<x_m$.
	Let $S_1=\{g_1,\ldots,g_s\}$ be the natural description of $\lambda_d(K)$.
	Notice that to prove the statement of the theorem, it is sufficient to find 
		$F\in \Pos^2_{\succeq 0}(\lambda_d(K))$ of degree $2k$ such that $F\notin T^2_{S_1,2k}$.
	By Claim 2 of Proposition \ref{trditev-protiprimer-interval-in-poltrak}, 
	there is $F\in \Pos^2_{\succeq 0}(\lambda_d(K))$ of degree 2 such 
	that $F\notin T^{2}_{S_1,2}$.
	This concludes the proof.
	\end{proof}
Theorem \ref{glavni} gives a characterization of the set $\Pos_{\succeq 0}^n(K)$ for unbounded sets $K$.
	
	\begin{theorem} \label{glavni}
		Suppose $K$ is an unbounded closed semialgebraic set in $\RR$ and $S$
	the natural description of $K$. 
	Then, for
	any $F\in \HH_n(\CC[x])$, the following are equivalent:
		\begin{enumerate}
			\item $F\in \Pos_{\succeq 0}^n(K)$.
			\item There exists a polynomial $h\in \RR[x]$ such that for every isolated point $w\in K$,
				$h(w)\neq 0$ and $h^2 F\in T^n_S$.
			\item For every point $w\in \CC$ there exists a polynomial $h\in \RR[x]$ 
				such that $h(w)\neq 0$ and $h^2 F\in T^n_S$.
		\end{enumerate}
	\end{theorem}
	
	\begin{proof}
		For the implication $(3)\Rightarrow (2)$ construct $h$ in the same way as in Remark \ref{drugi-dokaz}
		(replace the boundary of $K$ with the set of its isolated points). The implication 
		$(2)\Rightarrow (1)$ is trivial. The proof of
		direction $(1)\Rightarrow (3)$ is the same as the proof of Proposition \ref{pomozna}, just
		that we use \cite[Theorem 2.2]{K-M} for the $n=1$ case instead of \cite[Theorem 5.17]{Scheiderer4}.
	\end{proof}

\section{Generalizations of the results to curves}

In this section Theorem \ref{nasicenost-na-R} is generalized to curves in $\RR^n$.
A characterization of sets $S$ satisfying Theorem \ref{verzija-za-krivulje}.1 was proved by Scheiderer in \cite[Theorem 5.17]{Scheiderer3} and \cite[Corollary 4.4]{Scheiderer4}. 
Using the same method as in the proof of Theorem \ref{nasicenost-na-R} we obtain the
implication $1.\Rightarrow 2.$ of the following theorem.


\begin{theorem} \label{verzija-za-krivulje}
		Suppose $I$ is a prime ideal of $\RR[\underline x]$ with $\dim(\frac{\RR[\underline x]}{I})=1$ and let
	$\mathcal Z(I):=\left\{\underline x\in \RR^d\colon f(\underline x)=0 \;\text{ for every }\; f\in I\right\}$ be its vanishing set. Let
	$S:=\{g_1,\ldots,g_s\}$ be a finite subset of $\RR[\underline x]$ and 
	$K_S=\{\underline x\in \RR^d\colon g_1(\underline x)\geq 0,\ldots,g_s(\underline x)\geq 0\}$
	the associated semialgebraic set.
	Suppose	the set
	$K_S\cap \mathcal Z(I)$ is compact.
	Then the following are equivalent:
	\begin{enumerate}
		\item The quadratic module $M^1_S+I$ is saturated.
			\item The $n$-th quadratic module $M^n_S+M_n(I)$ is saturated for every $n\in \NN$.
	\end{enumerate}
\end{theorem}


An example of a non-singular curve is the unit circle. 
Theorem \ref{zvezna} has an equivalent version for the unit complex circle $\TT$ (see \cite{Rosenblum} or \cite{Popov}). By passing from complex numbers to pairs of real numbers and by Theorem \ref{verzija-za-krivulje}, we obtain a generalization of this equivalent version to an arbitrary semialgebraic set in the unit circle.
To explain this generalization we need some notation. Let us equip the set of $n\times n$ matrix Laurent polynomials $M_n(\CC\left[z,\frac{1}{z}\right])$ with an involution $A(z)^\ast:=\overline{A(\frac{1}{\overline z})}^T$. We denote by $\HH_n(\CC[z,\frac{1}{z}])$ 
the set of all $B\in M_n(\CC\left[z,\frac{1}{z}\right])$ such that $B^\ast=B$,  
and by $\sum M_n(\CC\left[z\right])^2$ 
the set of all finite sums of elements of the form $B^\ast B$ where $B\in M_n(\CC\left[z\right])$.
Let 
	$\cS=\left\{b_1,\ldots, b_s\right\}$ be a finite set
from $\HH_1(\CC\left[z,\frac{1}{z}\right])$
and
	$\cK_{\cS}=\left\{z\in \TT\colon b_j(z)\geq 0,\; j=1,\ldots,s\right\}$
the associated semialgebraic set.
Let the \textsl{$n$-th matrix quadratic module generated by $\cS$} in $\HH_n(\CC[z,\frac{1}{z}]$ be
	$$\cM_\cS^n := \{\tau_0+ \tau_1b_1+\ldots+\tau_s b_s\colon \tau_j\in 
			\sum M_n\left(\CC\left[z\right]\right)^2\;\text{for}\;j=0,\ldots,s	
			\}.$$
We write $\Pos_{\succeq 0}^n(\cK_\cS)$ for the set of elements from $\HH_n(\CC[z,\frac{1}{z}])$ 
which are positive semidefinite on $\cK_\cS$.



\begin{corollary}\label{posledica1}
	$\cM_\cS^n=\Pos_{\succeq 0}^n(\cK_\cS)$ iff $\cS$ satisfies the following conditions:
		\begin{enumerate}
			\item[(a)] For every boundary point $a\in \cK_\cS$ which is not isolated
				there exists $k\in \{1,\ldots,s\}$ such that $b_k(a)=0$
				and $\frac{db_k}{dz}(a)\neq 0$.
			\item[(b)] For every isolated point $a\in \cK_\cS$
				there exist $k,l\in \{1,\ldots,s\}$ such that $b_k(a)=b_l(a)=0$,
				$\frac{db_k}{dz}(a)\neq 0, \frac{db_l}{dz}(a)\neq 0$
				and $b_kb_l\leq 0$ on some neighborhood of $a$.
		\end{enumerate}
\end{corollary}

As an application of Corollary \ref{posledica1} we obtain the following improvement of 
Theorem \ref{glavni}:

\begin{corollary} \label{glavni-2}
		Suppose $K$ is an unbounded closed semialgebraic set in $\RR$ and $S$
	the natural description of $K$. 
	Then, for
	$F\in \HH_n(\CC[x])$, the following are equivalent:
		\begin{enumerate}
			\item $F\in \Pos_{\succeq 0}^n(K)$.
			\item For every $w\in \CC\setminus \RR$ there exists $k_w\in \NN\cup\{0\}$ such 
				that $$((x-\overline{w})(x-w))^{k_w} F\in M^n_S.$$
		\end{enumerate}
\end{corollary}

To prove Corollary \ref{glavni-2} we need some preliminaries.  M\"obius transformations that map $\RR\cup \{\infty\}$ 
bijectively into $\TT$ are exactly the maps of the form
	$$\lambda_{z_0,w_0}:\RR \cup \{\infty\} \to \TT,\quad \lambda_{z_0,w_0}(x):=
		z_0\frac{x-w_0}{x-\overline{w_0}},$$
where $z_0\in \TT$ and $w_0\in \CC\setminus \RR$. 
Notice that $\lambda_{z_0,w_0}^{-1}(x)=\frac{z\overline{w_0}-z_0w_0}{z-z_0}$.
If $F(x)$ is a matrix polynomial from $M_n(\CC[x])$, then 
$$\Lambda_{z_0,w_0,F}(z):=	((z-z_0)^\ast (z-z_0))^{\left\lceil \frac{\deg(F)}{2}\right\rceil} \cdot 
		F\left(\lambda^{-1}_{z_0,w_0}(z)	
			\right)$$
is a matrix polynomial from $M_n(\CC[z,\frac{1}{z}])$. Observe that 
$$F(x)=\left(\frac{(x-\overline{w_0})(x-w_0)}{4\cdot \Im(w_0)^2}\right)^{
		\left\lceil \frac{\deg(F)}{2}\right\rceil}\cdot
		\Lambda_{z_0,w_0,F}(\lambda_{z_0,w_0}(x)),$$
where $\Im(w_0)$ is the imaginary part of $w_0$.

\begin{proof}[Proof of Corollary \ref{glavni-2}]
	The non-trivial direction is $1.\Rightarrow 2.$ Choose $w_0\in \CC\setminus \RR$.	
	Observe that $\Lambda_{1,w_0,F}(z)$ belongs to the set $\Pos_{\succeq 0}^n(\cK_{w_0})$ where
	$\cK_{w_0}:=\Cl\left(\lambda_{1,w_0}(K)\right)$ and $\Cl(\cdot)$ is the closure operator. 
Let $S=\{g_1,\ldots,g_s\}$ be the natural description of $K$. Then
$\cS:=\{\Lambda_{1,w_0,g_1}(z),\ldots,\Lambda_{1,w_0,g_s}(z)\}$ satisfies the conditions of
Corollary \ref{posledica1} and hence $\Lambda_{1,w_0,F}\in \cM^n_{\cS}$.
Therefore
		$$\left(\frac{(x-\overline{w_0})(x-w_0)}{4\cdot\im(w_0)^2}\right)^{k_{w_0}} \cdot F(x)\in M^n_S,$$
	where $k_{w_0}\in \NN\cup\{0\}$ equals $k-\left\lceil \frac{\deg(F)}{2}\right\rceil$
	with $k$ being the degree of the summand of the highest degree in the 
	expression of	$\Lambda_{1,w_0,F}(z)$ as the element of $\cM^n_{\cS}$.
\end{proof}

\begin{remark}\label{preostali}
	By a similar but more technical proof we can show, that Corollary \ref{glavni-2}.2 is true for all 
	$w\in \CC\setminus K$, i.e., it is true also for $w\in \RR\setminus K$.
\end{remark}

\vspace{10 pt}
\noindent \textbf{Acknowledgment.}
	I would like to thank to my advisor Jaka Cimpri\v c for proposing the problem,
many helpful suggestions and the help in establishing Claim 2 of 
Proposition \ref{trditev-protiprimer-interval-in-poltrak}.

	I am also very grateful to the anonymous referee for a detailed reading of the previous and final versions of the manuscript and many suggestions for improvements.

\end{document}